\documentclass[a4paper,12pt]{article}
\usepackage{latexsym,amsmath,amsthm,amssymb}
\usepackage{a4wide}
\usepackage{hyperref}
\usepackage{marginnote}
\usepackage{color}
\usepackage{cite}
\hypersetup{
pdftitle={MT inequality}   
pdfauthor={Van Hoang Nguyen},
colorlinks = true,
linkcolor = magenta,
citecolor = blue,
}

\theoremstyle{plain}
\newtheorem{theorem}{Theorem}[section]

\newtheorem{proposition}[theorem]{Proposition}
\newtheorem{lemma}[theorem]{Lemma}

\theoremstyle{definition} 
\newtheorem{definition}[theorem]{Definition}

\theoremstyle{remark}

\renewcommand{\thefootnote}{\arabic{footnote}}

\def\R{\mathbb R}

\def\al{\alpha}
\def\om{\omega}
\def\Om{\Omega}
\def\be{\beta}
\def\de{\delta}
\def\De{\Delta} 

\def\lam{\lambda}
\def\ep{\epsilon}
\def\na{\nabla}
\def\pa{\partial}
\def\la{\langle} 
\def\ra{\rangle} 
\def\lt{\left}
\def\rt{\right}

\numberwithin{equation}{section}

\title{Extremal functions for the sharp Moser--Trudinger type inequalities in whole space $\mathbb R^N$}
\author{Van Hoang Nguyen\footnote{
Institut de Math\'ematiques de Toulouse, Universit\'e Paul Sabatier, 118 Route de Narbonne, 31062 Toulouse c\'edex 09, France.}
}

\begin{document}
\maketitle


\renewcommand{\thefootnote}{}

\footnote{Email: \href{mailto: Van Hoang Nguyen <van-hoang.nguyen@math.univ-toulouse.fr>}{van-hoang.nguyen@math.univ-toulouse.fr}}

\footnote{2010 \emph{Mathematics Subject Classification\text}: 46E35, 26D10.}

\footnote{\emph{Key words and phrases\text}: sharp Moser--Trudinger type inequality, normalized vanishing limit, normalized concentrating limit, maximizers.}

\renewcommand{\thefootnote}{\arabic{footnote}}
\setcounter{footnote}{0}

\begin{abstract}
This paper is devoted to study the sharp Moser--Trudinger type inequalities in whole space $\mathbb R^N$, $N \geq 2$ in more general case. We first compute explicitly the \emph{normalized vanishing limit} and the \emph{normalized concentrating limit} of the Moser--Trudinger type functional associated with our inequalities over all the \emph{normalized vanishing sequences} and the \emph{normalized concentrating sequences}, respectively. Exploiting these limits together with the concentration--compactness principle of Lions type, we give a proof of the exitence of maximizers for these Moser--Trudinger type inequalities. Our approach gives an alternative proof of the existence of maximizers for the Moser--Trudinger inequality and singular Moser--Trudinger inequality in whole space $\R^N$ due to Li and Ruf \cite{LiRuf2008} and Li and Yang \cite{LiYang}.
\end{abstract}

\section{Introduction}
Let $\Omega \subset \R^N$ $(N\geq 2)$ be a bounded smooth domain of $\R^N$ and $W_0^{1,N}(\Om)$ be the usual Sobolev space on $\Omega$, i.e., the completion of $C_0^\infty(\Om)$ under the Dirichlet norm $\lt(\int_\Om |\na u|^N dx\rt)^{\frac1N}$ here $\na u$ denoting the usual gradient of $u$. The famous Moser--Trudinger inequality \cite{Y1961,P1965,M1970,T1967} asserts that
\begin{equation}\label{eq:MTbounded}
\sup_{u\in W^{1,N}_0(\Om),\, \int_\Om |\na u|^N dx} \int_\Om e^{\alpha |u|^{\frac N{N-1}}} dx <\infty,
\end{equation}
for any $\alpha \leq \alpha_N :=N\omega_{N-1}^{\frac1{N-1}}$ where $\om_{N-1}$ denotes the area of unit sphere in $\R^N$. The inequality \eqref{eq:MTbounded} is sharp in sense that the supremum in \eqref{eq:MTbounded} will be infinite if $\al > \alpha_n$ although for each $u\in W_0^{1,N}(\Om)$ the integral is still finite. 

Using a rearrangement argument and a change of variables, Adimurthi and Sandeep \cite{AS07} generalized \eqref{eq:MTbounded} to a singular version as follows
\begin{equation}\label{eq:MTsingular}
\sup_{u\in W^{1,N}_0(\Om),\, \int_\Om |\na u|^N dx} \int_\Om |x|^{-N\beta} e^{\alpha |u|^{\frac N{N-1}}} dx <\infty,
\end{equation}
for any $\beta \in [0,1)$ and $\alpha \leq \alpha_{\beta,N} := (1-\beta) \al_N$. The inequality \eqref{eq:MTsingular} is sharp in the sense that the supremum in \eqref{eq:MTsingular} is infinite if $\al > \al_{\beta,N}$.

The Moser--Trudinger inequality was extended to entire space $\R^N$ by the pioneer works of Cao \cite{Cao}, do \'O \cite{doO97}, Ruf \cite{Ruf2005} and Li and Ruf \cite{LiRuf2008}. It was proved by Li and Ruf \cite{Ruf2005,LiRuf2008} that
\begin{equation}\label{eq:MTLiRuf}
\sup_{u\in W^{1,N}(\R^N),\, \int_{\R^N}(|\na u|^N + |u|^N) dx\leq 1} \int_{\R^N}\Phi_N(\al |u|^{\frac N{N-1}}) dx < \infty,
\end{equation}
for any $\al \leq \al_N$, where
\[
\Phi_N(t) = e^t -\sum_{k=0}^{N-2} \frac{t^k}{k!}.
\]
The inequality \eqref{eq:MTLiRuf} is sharp in the sense that the supremum in \eqref{eq:MTLiRuf} is infinite if $\al > \al_{N}$. The singular version of \eqref{eq:MTLiRuf} was later proved by Adimurthi and Yang \cite{AY10}, namely
\begin{equation}\label{eq:MTAY}
\sup_{u\in W^{1,N}(\R^N),\, \int_{\R^N}(|\na u|^N + |u|^N) dx\leq 1} \int_{\R^N}|x|^{-N\beta}\Phi_N(\al |u|^{\frac N{N-1}}) dx < \infty,
\end{equation}
for any $\beta\in [0,1)$ and $\al \leq \al_{\beta,N}$. Also, The inequality \eqref{eq:MTAY} is sharp in the sense that the supremum in \eqref{eq:MTAY} is infinite if $\al > \al_{\beta,N}$. When $\beta =0$ \eqref{eq:MTsingular} and \eqref{eq:MTAY} reduce to \eqref{eq:MTbounded} and \eqref{eq:MTLiRuf} respectively. It should be remarked here that the proof of \eqref{eq:MTAY} in \cite{AY10} is essentially based on the Young inequality while the proof of \eqref{eq:MTLiRuf} is based on the method of blow-up analysis. A new and simpler proof of \eqref{eq:MTLiRuf} and \eqref{eq:MTAY} was given in \cite{LamLu}. This new proof can be applied to obtain the sharp Moser--Trudinger inequality in entire Heisenberg group \cite{LamLuhei}. Such kind of singular Moser--Trudinger inequality is very important in analysis of partial differential equations, for examples, see \cite{Yang12,Yang12a,Yang12b}. We refer the interest reader to \cite{AD2004,doO2014*,doO2015,doO2016,LuZhu,Nguyen2017,Tintarev} for recent developments of \eqref{eq:MTbounded} and \eqref{eq:MTLiRuf}.

An interesting problem related to the Moser--Trudinger inequality is whether or not extremal functions exist. Existence of extremal functions for the Moser--Trudinger inequality \eqref{eq:MTbounded} was proved by Carleson and Chang \cite{CC1986} when $\Om$ is the unit ball, by Struwe \cite{Struwe} when $\Om$ is close to the ball in the sense of measure, by Flucher \cite{Flucher1992} and Lin \cite{Lin1996} when $\Om$ is a general smooth bounded domain, and by Li \cite{Li2001} for compact Riemannian surfaces. The extremal functions for \eqref{eq:MTLiRuf} was consider by Ruf \cite{Ruf2005} and Ishiwata \cite{Ishi} when $N=2$. It was proved that there exists $\ep_0 >0$ such that the supremum in \eqref{eq:MTLiRuf} is attained if $\ep_0 \leq \alpha \leq 4\pi$ while for $\al >0$ small enough, the supremum is not attained. If $N \geq 3$ the supremum in \eqref{eq:MTLiRuf} is attained for any $0 \leq \al \leq \alpha_N$ (see \cite{Ishi} for $0\leq \al < \alpha_N$ and \cite{LiRuf2008} for $\al =\alpha_N$). The extremal functions for \eqref{eq:MTAY} with $\beta \in (0,1)$ was recently proved by Li and Yang \cite{LiYang}. All proofs given in \cite{LiRuf2008,LiYang} for the existence of extremal functions for \eqref{eq:MTLiRuf} and \eqref{eq:MTAY} is based on the method of blow-up analysis. The method of blow-up analysis is now a standard method of dealing with the best Moser--Trudinger type inequalities. We refer the reader to the book \cite{Druet} and the articles \cite{AD2004,doO2014,doO2014*,doO2016,Li2001,Li2005,LiRuf2008,LiYang,LuYang,LuZhu,Nguyen2017,WY2012,Yang2017,Yang06,Yang06*,Yang09,Yang2015} for more details about this method.

In this paper, we give a new proof of the existence of extremal functions for \eqref{eq:MTLiRuf} and \eqref{eq:MTAY}. Our proof avoids the use of the method of blow-up analysis. Furthermore, we will consider a more general variational problem concerning to the Moser--Trudinger type inequalities. More precisely, we consider the following variational problem
\begin{equation}\label{eq:varproblem}
MT(N,\beta,F) = \sup_{u\in W^{1,N}(\R^N),\, \int_{\R^N} (|\na u|^N + |u|^N) dx \leq 1} \int_{\R^N} |x|^{-N\beta} F(u) dx,
\end{equation}
where $F$ is a nonnegative function and $\beta \in [0,1)$. The study of such a problem \eqref{eq:varproblem} is motivated by the work of De Figueiredo, do \'O, and Ruf \cite{deG02} where they studied a more genral variational problem concerning to the Moser--Trudinger type inequalities of the type
\[
\sup_{u\in W^{1,N}_0(\Om), \, \int_{\Om} |\na u|^N dx \leq 1} \int_\Om F(u) dx
\]
where $\Om$ is a smooth bounded domain of $\R^N$ and $F$ is a nonnegative function on $\R$ satisfying some suitable growth conditions. Following \cite{deG02}, we make the following assumptions on the growth condition of $F$ throughout this paper,
\begin{description}
\item (F1) $F \in C^1(\R)$.
\item (F2) $F$ is strictly increasing on $\R_+$ and $F(-t) = F(t)$, $t\in \R$.
\item (F3) $0 \leq F(t) \leq \Phi_N(\al_{\beta,N} |t|^{\frac N{N-1}})$ for all $t\in \R$.
\item (F4) There exist the limits $\lim\limits_{t\to \infty} e^{-\alpha_{\beta,N} t^{\frac N{N-1}}} F(t)$ and $C(F) = \lim\limits_{t\to 0} |t|^{-N} F(t)$.
\end{description}
We say that $F$ has subcritical growth if $\lim\limits_{t\to \infty} e^{-\alpha_{\beta,N} t^{\frac N{N-1}}} F(t) =0.$ Otherwise, we say that $F$ has critical growth, in this case we normalize to $\lim\limits_{t\to \infty} e^{-\alpha_{\beta,N} t^{\frac N{N-1}}} F(t) =1.$

Let $B_N$ denote the sharp constant in the following Gagliardo--Nirenberg--Sobolev inequality
\begin{equation}\label{eq:GNSineq}
B_N = \sup_{u\in W^{1,N}(\R^N),\, u\not\equiv 0} \, \frac{\|u\|_{2N}^{2N}}{\|\na u\|_N^N \|u\|_N^N}.
\end{equation}
A simple variational argument shows that $B_N$ is attained in $W^{1,N}(\R^N)$. Moreover, all maximizers are determined uniquely, up to a translation, dilation and multiple by a non-zero constant, by a spherically symmetric and non-increasing function.

Our first main result of this paper read as follow.
\begin{theorem}\label{Maximizers}
Let $N \geq 2$ and let $F$ be the function on $\R$ satisfying \text{\rm (F1)--(F4)}. Then the following conclusions holds.
\begin{description}
\item (i) If $\beta \in (0,1)$ and $F$ is subcritical then $MT(N,\beta,F)$ is attained. 
\item (ii) If $\beta \in (0,1)$ and $F$ is critical and satisfies the following condition
\begin{equation}\label{eq:lowercond}
F(t) \geq \Phi_N ( \alpha_{\beta,N} |t|^{\frac{N}{N-1}}) -\lambda |t|^N,
\end{equation}
with $\lambda < \alpha_{\beta,N}^{N-1}/(N-1)!$, then $MT(N,\beta,F)$ is attained.
\item (iii) Suppose that $\beta =0$ and $F$ is subcritical and has the following expression
\begin{equation}\label{eq:near0}
F(t) = C(F) |t|^N + \sum_{k=N}^{2(N-1)} C_k |t|^{\frac{Nk}{N-1}} + o(|t|^{2N})
\end{equation}
for $|t|$ small with $C_k \geq 0$ for any $k =N,\ldots,2N-2$. Then $MT(N,0,F)$ is attained if either $C_k >0$ for some $k =N,\ldots,2N-3$ or $C_k =0$ for any $k =N,\ldots,2N-3$ and $C_{2(N-1)} > C(F)/B_N$.
\item (iv) Suppose that $\beta =0$ and $F$ is critical and  satisfies \eqref{eq:lowercond} with $\lambda < \alpha_{N}^{N-1}/(N-1)!$ and \eqref{eq:near0} for $|t|$ small. Then $MT(N,0,F)$ is attained if either $C_k >0$ for some $k =N,\ldots,2N-3$ or $C_k =0$ for any $k =N,\ldots,2N-3$ and $C_{2(N-1)} > C(F)/B_N$.
\end{description}
\end{theorem}

The function $F(t) = \Phi_N(\al |t|^{\frac N{N-1}})$ with $\alpha \leq \alpha_{\beta,N}$ satisfies the conditions of Theorem \ref{Maximizers}, hence we recover the results of Ruf \cite{Ruf2005}, Li and Ruf \cite{LiRuf2008}, Ishiwata \cite{Ishi} and Li and Yang \cite{LiYang} for the Moser--Trudinger inequality and singular Moser--Trudinger inequality in $\R^N$. Theorem \ref{Maximizers} generalizes the results in \cite{deG02} for the unit ball to the wholes space $\R^N$. Comparing with the results in \cite{deG02} we see that the behavior of $F$ near zero \eqref{eq:near0} plays an important role in studying the existence of maximizers for $MT(N,0,F)$. Indeed, without this condition, nonexistence results can occur, for examples, when $N=2$ and $F(t)= e^{\alpha t^2} -1$ with $\alpha >0$ very small (see \cite{Ishi}). Another example is $N =3$ and $F(t) = \Phi_3(\al |t|^{3/2}) - \alpha^3|t|^{9/2}/6$ with $\al >0$ very small (the proof is completely similar with the one of Ishiwata \cite{Ishi}). The condition \eqref{eq:lowercond} with $\lam < \alpha_{\beta,N}^{N-1}/(N-1)!$ is used to exclude the concentrating behavior of the maximizer sequence for $MT(N,\beta,F)$. We do not know, at this time, that whether or not maximizer for $MT(N,\beta,F)$ exists without this condition. A similar open problem on $B_1$ was posed in \cite{deG02} (see section $2.5$ in that paper).

Let us explain the method used to prove Theorem \ref{Maximizers}. We take a maximizing sequence $u_n$ for $MT(N,\beta,F)$. Using a rearrangement argument based on P\'olya-Szeg\"o principle \cite{Brothers}, we can assume that $u_n$ is spherically symmetric and non-increasing function. According to a concentration--compactness principle of Lions type \cite{Lions1985} (see Lemma \ref{CCLions} below) we have one of following three possibles. Either $u_n$ is a normalized concentrating sequences or $u_n$ is a normalized vanishing sequences (see the precise definitions below) or $u_n$ converges weakly in $W^{1,N}(\R^N)$ to a non-zero function $u_0$. Computing explicitly the upper bounds of the Moser--Trudinger type functional
\begin{equation}\label{eq:MTfunctional}
J_F^\beta(u) = \int_{\R^N} |x|^{-N \beta}\, F(u) dx
\end{equation}
on the normalized concentrating sequences and the normalized vanishing sequences, we can exclude the concentrating and vanishing behavior of $u_n$. Using again the concentration--compactness principle, we can prove the existence of extremal functions for $MT(N,\beta,F)$ in Theorem \ref{Maximizers}.

We now make precisely the definitions of the normalized concentrating sequences and the normalized vanishing sequences following \cite{Ishi}.
\begin{definition}\label{NCSetNVS}
Let $\{u_n\}_n \subset W^{1,N}(\R^N)$ be a sequence such that $u_n \rightharpoonup u$ weakly in $W^{1,N}(\R^N)$. 
\begin{description}
\item (i) We say that $\{u_n\}_n$ is a normalized concentrating sequence ((NCS) in short) if $\|\na u_n\|_N^N + \|u_n\|_N^N =1$ for all $n$, $u = 0$ and $\lim\limits_{n\to \infty} \int_{B_R^c} \lt(|\na u_n|^N +  |u_n|^N\rt) dx =0,$ for any $R >0$. A (NCS) consisting of radially symmetric functions is called a radially symmetric normalized concentrating sequence ((RNCS) in short).
\item (ii) We say that $\{u_n\}_n$ is a normalized vanishing sequence ((NVS) in short) if $\|\na u_n\|_N^N + \|u_n\|_N^N =1$ for all $n$, $u = 0$ and $\lim\limits_{n\to \infty} \int_{B_R} |x|^{-N\beta} F(u_n) dx = 0,$ for any $R >0$. A (NVS) consisting of radially symmetric functions is called a radially symmetric normalized vanishing sequence ((RNVS) in short).
\end{description}
\end{definition}

Next we introduce obstacle values for the compactness of maximizing sequences
\begin{definition}\label{NCSandNVSlimit}
$\bullet$ The number
\[
d_{ncl}(N,\beta,F) = \sup_{\{u_n\}: {\rm (RNCS)}}\limsup_{n\to \infty} J_F^\beta(u_n),
\]
is called a normalized concentration limit.

$\bullet$ The number 
\[
d_{nvl}(N,\beta,F) = \sup_{\{u_n\}: {\rm (RNVS)}}\limsup_{n\to \infty} J_F^\beta(u_n),
\]
is called a normalized vanishing limit.
\end{definition}

We will compute explicitly two limits in the following two theorems. The first one give us the normalized vanishing limit $d_{nvl}(N,\beta,F)$.
\begin{theorem}\label{NVSlimit}
Let $F$ be a function on $\R$ satisfying the conditions $\text{\rm (F1)--(F4)}$. Then we have
\[
d_{nvl}(N,\beta,F) = 
\begin{cases}
0&\mbox{if $\beta \in (0,1)$,}\\
C(F) &\mbox{if $\beta =0$}.
\end{cases}
\]
\end{theorem}

To state the normalized concentrating limit $d_{ncl}(N,\beta,F)$, let us give some notation. Let $G$ be the distributional solution of the equation
\begin{equation}\label{eq:Green}
-\De_N G + G^{N-1} = \de_0
\end{equation}
on $\R^N$ where $-\De_N$ denotes $N-$Laplacian operator on $\R^N$ and $\de_0$ denotes the Diract measure at the origin. It is well known that $G \in L^N(\R^N) \cap W^{1,N}_{\rm loc}(\R^N\setminus\{0\})$ (e.g., see \cite{LiRuf2008}). Moreover, $G$ is a spherically symmetric and strictly decreasing function and has the following expression
\begin{equation}\label{eq:formG}
G(x) = -\frac N{\al_N} \ln |x| + A_0 + O(|x|^{N} \ln^{N-1}|x|),
\end{equation}
when $x \to 0$ with $A_0$ is constant, and
\begin{equation}\label{eq:formG'}
|\na G(x)| = \frac{N}{\alpha_N} \frac1{|x|} + O(|x|^{N-1} \ln^{N-1} |x|),\qquad x \to 0.
\end{equation}
The normalized concentrating limit $d_{ncl}(N,\beta,F)$ is given in the following theorem.
\begin{theorem}\label{NCSlimit}
Let $F$ be a function on $\R$ satisfying the conditions $\text{\rm (F1)--(F4)}$. Then we have
\[
d_{ncl}(N,\beta,F) =
\begin{cases}
0&\mbox{if $F$ is subcritical,}\\
\frac1{1-\beta}|B_1| e^{(1-\beta)\alpha_N A_0 + 1 + \frac12 +\cdots+ \frac 1{N-1}}&\mbox{if $F$ is critical,}
\end{cases}
\]
where $A_0$ is constant given by \eqref{eq:formG}.
\end{theorem}
 
From Theorem \ref{NVSlimit} and Theorem \ref{NCSlimit}, we see that the normalized vanishing limit and normalized concentrating limit depend on the behavior of the function $F$ near zero and near infinity respectively. In particular, the normalized concentrating sequences have no role in studying $MT(N,\beta,F)$ when $F$ is subcritical. The proof of Theorem \ref{NVSlimit} is elementary by using the assumption (F4). The proof of Theorem \ref{NCSlimit} is more complicate. We follow the argument of Ruf \cite{Ruf2005} when $N=2$ by combining the arguments in \cite{deG91,deG02}. Accidentally, we correct a gap in \cite{Ruf2005}. In that paper, Ruf proved that the normalized concentrating limits $d_{ncl}(2,0,F) =\pi e$ with $F(t) = e^{4\pi t^2} -1$. However, we know that $G(r) = K_0(r)/(2\pi)$ where $K_0$ denotes the modified Bessel function of second kind, hence $A_0 = (\ln 2 -\gamma)/(2\pi)$ with $\gamma$ being the Euler--Mascheroni constant. Hence this limit is $4\pi e^{1-2\gamma}$ which is strict larger than $e\pi$.

The normalized concentrating limit in Theorem \ref{NCSlimit} will be used in \cite{Nguyen17} to study the sharp Moser--Trudinger inequality in $\R^N$ under the equivalent constraints 
\begin{equation}\label{eq;constraint}
S_{a,b} = \{u \in W^{1,N}(\R^N)\, :\, \|\na u\|_N^a + \|u\|_N^b =1\},
\end{equation}
where $N \geq 2$ and $a, b >0$. Let us denote for $a, b$ and $\al >0$
\begin{equation}\label{eq:supremum}
d_{N,\alpha}(a,b) = \sup_{u\in S_{a,b}} \int_{\R^N} \Phi_N(\al |u|^{\frac N{N-1}}) dx.
\end{equation}
It is easy to see that $d_{N,\al}(a,b) < \infty$ if $\al < \al_N$ for any $a,b >0$. It was proved by Lam, Lu and Zhang \cite{LamLuZhang} that $d_{N,\al_N}(a,b) < \infty$ if and only if $b\leq N$. In \cite{doO16}, do \'O, Sani and Tarsi studied the effect of $a$ and $b$ to the attainability of $d_{N,\al_N}(a,b)$ and proved in the case $0< b< N$ and $N \geq 2$ that $d_{N,\al_N}(a,b)$ is attained if $a > N/(N-1)$. The case $b =N$ is excluded in \cite{doO16} to avoid the concentrating behavior of the Moser--Trudinger type functional in the critical case. Indeed, when $0< b< N$ the concentrating behavior does not occur. The case $b =N$ was treated in \cite{Nguyen17} by the author. Exploiting Theorem \ref{NCSlimit} and an appropriate change of functions, the author proves that $d_{N,\al_N}(a,N)$ is attained if 
\[
\frac{N}{N-1} < a < N + \frac{2 N^{N+1}}{(N-1)!}e^{-\al_N A_0 -1 -\frac12 -\cdots -\frac1{N-1}},
\]
where $A_0$ is constant given by \eqref{eq:formG}. In this direction, we refer the interest reader to the paper of Lam, Lu and Zhang \cite{LamLuZhanga} for more general results (still in the subcritical case).


The rest of this paper is organized as follows. In the next section, we recall some fact used in our proofs. The proofs of Theorem \ref{NVSlimit} and Theorem \ref{NCSlimit} are given in section \S3 and \S4 respectively. In section \S5, we prove Theorem \ref{Maximizers}. Finally, we prove an auxiliary variational problem which is used in the proof of Theorem \ref{NCSlimit} in section \S6. 


\section{Preliminaries}
In this section, we introduce some useful results that will be used in our proofs. We first recall the definition of the decreasing symmetric rearrangement and some its properties. Let $\Om \subset \R^N$, $N \geq 2$ be a measurable set. We denote by $\Om^\sharp$ the open ball $B_R$ centered at origin of radius $R >0$ such that $|\Om| =|B_R|$. Let $u :\Om \to \R$ be a measurable function that vanishes at infinity, i.e., for any $t>0$ the set $\{|u| >t\}$ has finite measure. For such a function $u$, its distribution function $\mu_u$ is defined by
\[
\mu_u(t) = |\{x\in \Om\, :\, |u(x)| >t\}|\qquad t >0,
\]
and its decreasing rearrangement function $u^\star$ is defined by
\[
u^*(t) = \inf\{s \geq 0\, :\, \mu_u(s) < t\}.
\]
Note that $u^*$ is a decreasing function on $[0, |\Om|]$. Moreover, the decreasing symmetric rearrangement function $u^\sharp :\Om^\sharp \to \R$ of $u$ is defined by 
\[
u^\sharp(x) = u^*(|B_1| |x|^N)\qquad x \in \Om^\sharp.
\]
Note that $u$ and $u^\sharp$ have the same distribution function, hence $\int_\Om |u|^p dx = \int_{\Om^\sharp} |u^\sharp|^p dx$ for any $p\geq 1$. Moreover, we have the following P\'olya--Szeg\"o principle \cite{Brothers}
\[
\int_\Om |\na u|^p dx \geq \int_{\Om^\sharp} |\na u^\sharp|^p dx,
\]
for any $p \in [1,\infty)$. Consequently, if $u \in W^{1,p}(\R^N)$ then $u^\sharp \in W^{1,p}(\R^N)$. 

We also need the following radial lemma.
\begin{lemma}\label{radial}
If $u\in W^{1,N}(\mathbb R^N)$ be a radial function, then 
\[
|u(x)| \leq \frac{C \|u\|_{W^{1,N}(\R^N)}}{|x|^{\frac{N-1}N}},
\]
with $C$ depends only on $N$.
\end{lemma}

In our proofs below, we will use frequently the following elementary inequality which follows by the convexity of the function $t \mapsto t^r, r>1$ in $\R_+$.
\begin{lemma}\label{elementary}
Let $r >1$ then for any $a, b\geq 0$ and $\ep >0$, we have
\begin{equation}\label{eq:elementary}
(a+ b)^r \leq (1+ \ep) a^r + (1-(1+\ep)^{\frac 1{1-r}})^{1-r} b^r.
\end{equation}
\end{lemma}

We next collect a concentration--compactness priciple of Lions type \cite{Lions1985} for the sequence of the spherically symmetric and non-increasing functions in $W^{1,N}(\R^N)$.
\begin{lemma}\label{CCLions}
Let $\{u_n\}$ be a sequence of spherically symmetric and non-increasing functions in $W^{1,N}(\R^N)$ with $\|\na u_n\|_N^N + \|u_n\|_N^N =1$ and $u_n \rightharpoonup u_0$ weakly in $W^{1,N}(\R^N)$. Then one of the following conclusions holds.
\begin{description}
\item (i) $u_n$ is (RNCS).
\item (ii) $u_n$ is (RNVS).
\item (iii) $u_0\not=0$ and 
\[
\limsup_{n\to \infty} \int_{\R^N} |x|^{-N\beta} \Phi_N(\alpha_{\beta,N} p u_n^{\frac N{N-1}}) dx < \infty
\]
for any $p < (1 -\|u_0\|_N^N -\|\na u_0\|_N^N)^{-1/(N-1)}$.
\end{description}
\end{lemma}
\begin{proof}
Suppose that $u_0 \equiv 0$. If $u_n$ is not (RNCS), then there exist $R_0 > 0$ and $\delta \in (0,1]$ such that 
\[
\lim_{n\to \infty} \int_{B_{R_0}^c} (|\na u_n|^N + |u_n|^N) dx \geq \de.
\]
From Lemma \ref{radial}, we have $u_n(r) \leq C R_0^{-\frac{N-1}N}$ for any $r \geq R_0$ with $C$ depending only on $N$ (here $u_n(r)$ means the value of $u_n$ at $x$ with $|x| =r$). Hence $F(u_n(r)) \leq C u_n(r)^N$ for any $r \geq R_0$ with $C$ depending only on $N$, $R_0$ and $F$. We have following two cases: 

$\bullet$ \emph{Case 1:} $\lim\limits_{n\to\infty} \int_{B_{R_0}^c} |u_n|^N dx =0$. Since 
\[
\int_{B_R\setminus B_{R_0}} |x|^{-N\beta} F(u_n) dx \leq C R_0^{-N\beta}\int_{B_{R_0}^c} |u_n|^N dx \to 0
\]
as $n \to \infty$, hence it is enough to prove
\begin{equation}\label{eq:R0}
\lim_{n\to\infty} \int_{B_{R_0}} |x|^{-N\beta} F(u_n) dx =0.
\end{equation}
Define $v_n(r) = u_n(r) -u_n(R_0)$, then $v_n \in W^{1,N}_0(B_{R_0})$ and
\[
\lim_{n\to \infty} \int_{B_{R_0}} |\na v_n|^N dx =\lim_{n \to \infty} \int_{B_{R_0}} |\na u_n|^N dx \leq 1 -\de < 1.
\]
Hence, there exist $a \in (1-\de,1)$ and $n_0$ such that $\int_{B_{R_0}} |\na v_n|^N dx \leq a < 1$ for any $n \geq n_0$. By Lemma \ref{elementary}, we have
\[
u_n^{\frac N{N-1}} =\lt(v_n + u_n(R_0)\rt)^{\frac N{N-1}} \leq (1+ \ep) v_n^{\frac{N}{N-1}} + \lt(1-(1+ \ep)^{1-N}\rt)^{\frac1{1-N}} u_n(R_0)^{\frac N{N-1}}.
\]
Therefore,
\begin{align*}
\int_{B_{R_0}} |x|^{-N \beta} F(u_n) dx &\leq \int_{B_{R_0}} e^{\al_{\beta,N} |u_n|^{\frac N{N-1}}} |x|^{-N\beta} dx\\
&\leq C(N,R_0,\ep) \int_{B_{R_0}} e^{\al_{\beta,N}(1+ \ep) |v_n|^{\frac N{N-1}}} |x|^{-N\beta} dx,
\end{align*}
with $C(N,R_0,\ep)$ depending only on $N,R_0$ and $\ep$. Choosing $\ep >0$ small enough such that $(1+ \ep) a^{1/(N-1)} < 1$, we obtain from the singular Moser--Trudinger inequality that $F(u_n)$ is bounded in $L^r(B_{R_0}, |x|^{-N\beta} dx)$ for some $r > 1$. This together with $u_n\to u_0 = 0$ a.e in $\R^N$ implies \eqref{eq:R0}.

$\bullet$ \emph{Case 2:} $\lim\limits_{n\to\infty} \int_{B_{R_0}^c} |u_n|^N dx =\sigma >0.$ In this case, we have
\[
\liminf_{n\to \infty} \int_{\R^N} |u_n|^N dx \geq \sigma >0,
\]
and then
\[
\limsup_{n\to \infty} \int_{B_R} |\na u_n|^N dx \leq 1 -\sigma < 1,
\]
for any $R >0$. Repeating the proof of \eqref{eq:R0} with $R_0$ replaced by $R$, we get
\[
\lim_{n\to\infty} \int_{B_{R}} |x|^{-N\beta} F(u_n) dx =0.
\]
Hence $u_n$ is (RNVS).

If $u_0\not\equiv 0$, then the conclusion follows from \cite[Theorem $1.1$]{doO2014} for $\beta =0$ and from \cite[Theorem $1.1$]{doO2015*} for $\beta\in (0,1)$.
\end{proof}

\section{Proof of Theorem \ref{NVSlimit}}
In this section, we compute the normalized vanishing limit $d_{nvl}$ in Theorem \ref{NVSlimit}. We divide our proof into two cases following $\beta =0$ or $\beta \in (0,1)$.

$\bullet$ \emph{Case 1:} $\beta \in (0,1)$. Let $u_n$ be an arbitrary (RNVS). By Lemma \ref{radial}, we have $u_n(R) \leq C R^{-(N-1)/N}$ for any $R >0$ with $C$ depending only on $N$. The assumtion $\text{\rm (F4)}$ on $F$ implies that $F(t) = (C(F)+o_t(1)) |t|^N$ as $t\to 0$. Thus, for any $R >0$ we get
\[
\int_{B_R^c}|x|^{-\beta N} F(u_n) dx \leq R^{-\beta N} \int_{B_R^c} F(u_n) dx =R^{-\beta N} (C(F) + o_R(1)) \int_{B_R^c} |u_n|^N dx,
\]
as $R \to \infty$. Hence
\begin{equation}\label{eq:outsubvanishing}
\lim_{R\to\infty} \limsup_{n\to\infty}\int_{B_R^c}|x|^{-\beta N} F(u_n) dx =0.
\end{equation}
Since $u_n$ is a (RNVS) then
\[
\lim_{n\to \infty} \int_{B_R} |x|^{-\beta N} F(u_n) dx =0,
\]
for any $R >0$. This together with \eqref{eq:outsubvanishing} implies
\[
\lim_{n\to \infty} \int_{\R^N} |x|^{-N\beta} F(u_n) dx =0.
\]
Since $u_n$ is an arbitrary (RNVS), hence $d_{nvl}(N,\beta,F) =0$.

$\bullet$ \emph{Case 2:} $\beta =0$. Let $u_n$ be an arbitrary (RNVS). We know that $u_n(R) \leq C R^{-(N-1)/N}$ for any $R >0$ with $C$ depending only on $N$. Hence 
\[
F(u_n(r)) = (C(F) + o_R(1)) |u_n(r)|^N
\]
for any $r \geq R$ when $R\to\infty$ by assumption (F4) on $F$. Hence
\[
\int_{B_R^c} F(u_n) dx =(C(F) + o_R(1)) \int_{B_R^c} |u_n|^N dx \leq C(F) + o_R(1),
\]
which implies
\begin{equation}\label{eq:outsubvanish}
\lim_{R\to\infty} \limsup_{n\to \infty} \int_{B_R^c} F(u_n) dx \leq C(F).
\end{equation}
Since $u_n$ is a (RNVS) then
\[
\lim_{n\to \infty} \int_{B_R} F(u_n) dx =0,
\]
for any $R >0$. This together with \eqref{eq:outsubvanishing} implies
\[
\lim_{n\to \infty} \int_{\R^N} F(u_n) dx \leq C(F).
\]
Since $u_n$ is an arbitrary (RNVS), hence $d_{nvl}(N,0,F) \leq C(F)$.

It remains to show that $d_{nvl}(N,0,F) \geq C(F)$. Let $\phi$ be a smooth, compactly supported, spherically symmetric and non-increasing function in $\R^N$ with $\|\na \phi\|_N = \|\phi\|_N =1$, and let $\lam_n$ be a sequence of nonnegative numbers such that $\lam_n \to 0$ as $n \to \infty$. Define $\psi_n(x) = \lam_n \phi (\lam_n x)$, then $\|\na \psi_n\|_N^N =\lambda_n^N$ and $\|\psi_n\|_N^N =1$. Define $u_n = \psi_n /(1+ \lam_n^N)^{1/N}$, then $u_n$ is a (RNVS) as $n\to \infty$ (see Example $2.1$ in \cite{Ishi}). Note that $u_n \to 0$ uniformly in $\R^N$ as $n\to \infty$. By assumption (F4) on $F$, we have $F(u_n) = (C(F) + o_n(1)) |u_n|^N$ as $n\to \infty$, and hence
\[
\int_{\R^N} F(u_n) dx = (C(F) + o_n(1))\int_{\R^N} |u_n|^N dx = (C(F) + o_n(1)) \frac{1}{1+ \lam_n^N}.
\]
Letting $n\to\infty$, we get
\[
d_{nvl}(N,0,F) \geq C(F).
\]
This finishes our proof.

\section{Proof of Theorem \ref{NCSlimit}}
We follow the argument in \cite{deG02,Ruf2005}. Let $u_n$ be an arbitrary (RNCS). We define the function $w_n$ on $\R$ by
\begin{equation}\label{eq:changefunct}
w_n(t) = [(1-\beta)N \omega_{N-1}^{1/(N-1)}]^{(N-1)/N}\, u_n(e^{-t/((1-\beta)N)}).
\end{equation}
Then $w_n$ is a non-decreasing function with $\lim\limits_{t\to -\infty} w_n(t) = 0$, and
\begin{equation}\label{eq:relationnabla}
\int_{\R} |w_n'(t)|^N dt = \int_{\R^N} |\na u_n|^N dx,
\end{equation}
\begin{equation}\label{eq:relationNnorm}
\int_{\R} |w_n(t)|^N e^{-t/(1-\beta)} dt =(1-\beta)^N N^N \int_{\R^N} |u_n|^N dx,
\end{equation}
and
\begin{equation}\label{eq:relationF}
\int_{\R^N}|x|^{-N\beta} F(u_n) dx = \frac{|B_1|}{1-\beta} \int_{\R} F\lt(\alpha_{\beta,N}^{\frac{1-N}N} w_n(t)\rt) e^{-t} dt.
\end{equation}

Since $u_n$ is (RNCS) then we have
\begin{equation}\label{eq:concentrationcond}
\lim_{n\to\infty} \int_{-\infty}^A \lt(|w_n'(t)|^N + \frac1{(1-\beta)^NN^N} |w_n(t)|^N e^{-t/(1-\beta)}\rt) dt =0,
\end{equation}
for any $A\in \R$. Also, by $u_n \rightharpoonup 0$ weakly in $W^{1,N}(\R^N)$ and $u_n$ is spherically symmetric and non-increasing, then $u_n \to 0$ a.e. in $\R^N$. Consequently, we get $w_n \to 0$ a.e. in $\R$. We first claim that
\begin{equation}\label{eq:claimsub}
\lim_{n\to\infty} \int_{-\infty}^A F\lt(\alpha_{\beta,N}^{\frac{1-N}N} w_n(t)\rt) e^{-t} dt =0
\end{equation}
for any $A \in \R$. Indeed, for a fixed $A$, we have $0\leq w_n(t) \leq w_n(A) \to 0$ as $n\to \infty$ for any $t\leq A$, hence
\[
F\lt(\alpha_{\beta,N}^{\frac{1-N}N} w_n(t)\rt) =(C(F) + o_n(1)) \al_{\beta,N}^{1-N} w_n(t)^{N}
\]
for any $t\leq A$ as $n\to \infty$. By this, we get
\begin{align*}
\int_{-\infty}^A F\lt(\alpha_{\beta,N}^{\frac{1-N}N} w_n(t)\rt) e^{-t} dt &= (C(F)+o_n(1)) \al_{\beta,N}^{1-N} \int_{-\infty}^A w_n(t)^N e^{-t} dt\\
&\leq (C(F)+o_n(1)) \al_{\beta,N}^{1-N} e^{\frac{\beta A}{1-\beta}} \int_{-\infty}^A w_n(t)^N e^{-\frac{t}{1-\beta}} dt.
\end{align*}
This together with \eqref{eq:concentrationcond} proves our claim \eqref{eq:claimsub}. We divide our proof into two cases following the subcriticality and criticality of $F$.


\subsection{$F$ is subcritical}
We first consider the case when $F$ is subcritical. In this case, we need to prove that
\begin{equation}\label{eq:suff}
\lim_{n\to \infty} \int_{\R} F\lt(\alpha_{\beta,N}^{\frac{1-N}N} w_n(t)\rt) e^{-t} dt = 0.
\end{equation}
Since $w_n(e)^{\frac N{N-1}} \to 0$ as $n\to \infty$, then there exists $n_0$ such that 
\[
w_n(e)^{\frac N{N-1}} < e -2 \ln e = e-2,
\]
for any $n\geq n_0$. We have two following cases:

$\bullet$ \emph{Case $1$}: $w_n(t)^{\frac N{N-1}} < t- 2\ln t$ for any $n\geq n_0$ and $t\geq e$. For any $A \geq e$, we have
\[
\int_A^\infty F\lt(\alpha_{\beta,N}^{\frac{1-N}N} w_n(t)\rt) e^{-t} dt \leq \int_A^\infty t^{-2} dt = \frac1A,
\]
for any $n\geq n_0$. Thus
\begin{align*}
\int_{\R} F\lt(\alpha_{\beta,N}^{\frac{1-N}N} w_n(t)\rt) e^{-t} dt &= \int_{-\infty}^A F\lt(\alpha_{\beta,N}^{\frac{1-N}N} w_n(t)\rt) e^{-t} dt + \int_A^\infty F\lt(\alpha_{\beta,N}^{\frac{1-N}N} w_n(t)\rt) e^{-t} dt\\
&\leq \int_{-\infty}^A F\lt(\alpha_{\beta,N}^{\frac{1-N}N} w_n(t)\rt) e^{-t} dt + \frac1A.
\end{align*}
Letting $n\to \infty$, $A \to \infty$ and using the claim \eqref{eq:claimsub}, we get \eqref{eq:suff}.

$\bullet$ \emph{Case $2$}: There exists $t \geq e$ such that $w_n(t)^{\frac N{N-1}} = t -2 \ln t$. Let $a_n$ be the first $t\geq e$ such that this equality holds. Using H\"older inequality, we have
\[
w_n(a_n) -w_n(0) = \int_0^{a_n} w_n'(t) dt \leq a_n^{\frac{N-1}N} \lt(\int_0^{a_n} |w_n'(t)|^N dt\rt)^{\frac 1N}.
\]
Dividing both sides by $a_n^{\frac{N-1}N}$, we get
\[
\lt(1 -\frac{2\ln a_n}{a_n}\rt)^{\frac{N-1}N} -\frac{w_n(0)}{a_n^{\frac{N-1}N}} \leq \lt(\int_0^{a_n} |w_n'(t)|^N dt\rt)^{\frac 1N}.
\]
This implies that $a_n\to \infty$ since $w_n(0) \to 0$ as $n\to \infty$ and \eqref{eq:concentrationcond}. 

For a fixed $A \geq e$, we have $a_n \geq A$ for $n$ large enough. The definition of $a_n$ implies that $w_n(t)^{\frac N{N-1}} \leq t -2 \ln t$ for any $t \in (A,a_n)$. Hence
\[
\int_{A}^{a_n}F\lt(\alpha_{\beta,N}^{\frac{1-N}N} w_n(t)\rt) e^{-t} dt \leq \int_A^{a_n} t^{-2}dt = \frac 1A -\frac1{a_n}.
\]
On $[a_n, \infty)$ we have $w_n(t)^{\frac N{N-1}} \geq a_n - 2\ln a_n \to \infty$. Hence the subcriticality of $F$ implies that
\[
F\lt(\alpha_{\beta,N}^{\frac{1-N}N} w_n(t)\rt) = o_n(1) e^{w_n(t)^{\frac N{N-1}}} =o_n(1) \Phi_N(w_n(t)^{\frac N{N-1}}),
\]
for $t \geq a_n$ as $n\to \infty$. Integrating on $[a_n,\infty)$ and using the singular Moser--Trudinger inequality \eqref{eq:MTAY} and \eqref{eq:relationF} for $F(t) = \Phi_N(\al_{\beta,N} |t|^{\frac N{N-1}})$, we get
\begin{align*}
\int_{a_n}^\infty F\lt(\alpha_{\beta,N}^{\frac{1-N}N} w_n(t)\rt) e^{-t} dt& = o_n(1)\int_{a_n}^\infty \Phi_N(w_n(t)^{\frac N{N-1}}) e^{-t} dt\\
&\leq o_n(1) \int_0^\infty \Phi_N(w_n(t)^{\frac N{N-1}}) e^{-t} dt\\
&=o_n(1) \frac{1-\beta}{|B_1|} \int_{\R^N} |x|^{-N\beta} \Phi_N(\alpha_{\beta,N} |u_n|^{\frac N{N-1}}) dx\\
&= o_n(1).
\end{align*}
Thus
\begin{align*}
\int_{\R} F\lt(\alpha_{\beta,N}^{\frac{1-N}N} w_n(t)\rt) e^{-t} dt &= \int_{-\infty}^A F\lt(\alpha_{\beta,N}^{\frac{1-N}N} w_n(t)\rt) e^{-t} dt + \int_{A}^{a_n} F\lt(\alpha_{\beta,N}^{\frac{1-N}N} w_n(t)\rt) e^{-t} dt\\
&\qquad + \int_{a_n}^\infty F\lt(\alpha_{\beta,N}^{\frac{1-N}N} w_n(t)\rt) e^{-t} dt\\
&\leq \int_{-\infty}^A F\lt(\alpha_{\beta,N}^{\frac{1-N}N} w_n(t)\rt) e^{-t} dt + \frac1A -\frac1{a_n} + o_n(1). 
\end{align*}
Letting $n\to \infty$ and using \eqref{eq:claimsub}, and then letting $A \to \infty$ we obtain \eqref{eq:suff}.

Consequently, \eqref{eq:suff} holds. Using \eqref{eq:relationF}, we get
\[
\lim_{n\to\infty} \int_{\R^N} |x|^{-\beta N} F(u_n) dx =0.
\]
Since $u_n$ is an arbitrary (RNCS), then $d_{ncl}(N,\beta, F) =0$.


\subsection{$F$ is critical}
We first show that
\begin{equation}\label{eq:lower}
\lim_{n\to \infty} \int_{\R} F\lt(\alpha_{\beta,N}^{\frac{1-N}N} w_n(t)\rt) e^{-t} dt \leq e^{(1-\beta)\alpha_N A_0 + 1 + \frac12 + \cdots + \frac1{N-1}}.
\end{equation}
We can assume, in addition, that
\begin{equation}\label{eq:extracond}
\lim_{n\to \infty} \int_{\R} F\lt(\alpha_{\beta,N}^{\frac{1-N}N} w_n(t)\rt) e^{-t} dt > 0.
\end{equation}
Indeed, if otherwise there is nothing to prove. Since $w_n(e) \to 0$ as $n\to \infty$, we then have $w_n(e)^{\frac N{N-1}} < e -2 \ln e$ for any $n\geq n_0$ for some $n_0$. Let $a_n$ be the first $t\geq e$ such that $w_n(t)^{\frac N{N-1}} = t- 2\ln t$. Such an $a_n$ exists since otherwise we have
\[
w_n(t)^{\frac N{N-1}} < t -2 \ln t,\qquad\forall\, t \geq e.
\]
Thus for $A \geq e$ we have
\[
\int_A^\infty F\lt(\alpha_{\beta,N}^{\frac{1-N}N} w_n(t)\rt) e^{-t} dt \leq \int_A^\infty t^{-2} dt = \frac 1A,
\]
and hence
\[
\int_{\R} F\lt(\alpha_{\beta,N}^{\frac{1-N}N} w_n(t)\rt) e^{-t} dt \leq \int_{-\infty}^A F\lt(\alpha_{\beta,N}^{\frac{1-N}N} w_n(t)\rt) e^{-t} dt + \frac1A.
\]
Letting $n\to \infty$, using \eqref{eq:claimsub} and then letting $A\to \infty$, we get
\[
\lim_{n\to \infty} \int_{\R} F\lt(\alpha_{\beta,N}^{\frac{1-N}N} w_n(t)\rt) e^{-t} dt = 0,
\]
which contradicts with \eqref{eq:extracond}. Thus $a_n$ exists.

Repeating the argument in Case $2$ in the previous subcritical case, we get $a_n \to \infty$. We next claim that

\begin{equation}\label{eq:claim0}
\lim_{n\to \infty} \int_{-\infty}^{a_n} F\lt(\alpha_{\beta,N}^{\frac{1-N}N} w_n(t)\rt) e^{-t} dt = 0.
\end{equation}
Indeed, for any $A \geq e$, we have $a_n > A$ for $n$ large enough, and $w_n(t)^{\frac{N}{N-1}} < t -2 \ln t$ for any $t \in [A,a_n)$, hence
\begin{align*}
\int_{-\infty}^{a_n} F\lt(\alpha_{\beta,N}^{\frac{1-N}N} w_n(t)\rt) e^{-t} dt &= \int_{-\infty}^{A} F\lt(\alpha_{\beta,N}^{\frac{1-N}N} w_n(t)\rt) e^{-t} dt + \int_{A}^{a_n} F\lt(\alpha_{\beta,N}^{\frac{1-N}N} w_n(t)\rt) e^{-t} dt\\
&\leq \int_{-\infty}^{A} F\lt(\alpha_{\beta,N}^{\frac{1-N}N} w_n(t)\rt) e^{-t} dt + \frac1A -\frac1{a_n}.
\end{align*}
Letting $n\to \infty$, using \eqref{eq:claimsub} and then letting $A\to \infty$, we obtain \eqref{eq:claim0}.

By \eqref{eq:claim0}, it is enough to prove
\begin{equation}\label{eq:enough}
\lim_{n\to \infty} \int_{a_n}^\infty F\lt(\alpha_{\beta,N}^{\frac{1-N}N} w_n(t)\rt) e^{-t} dt \leq e^{(1-\beta)\alpha_N A_0 + 1 + \frac12 + \cdots + \frac1{N-1}}
\end{equation}
in order to verify \eqref{eq:lower}. On $[a_n,\infty)$, we have 
\[
w_n(t)^{\frac N{N-1}} \geq w_n(a_n)^{\frac{N}{N-1}} = a_n -2\ln a_n\to \infty.
\]
The criticality of $F$ implies that
\[
F\lt(\alpha_{\beta,N}^{\frac{1-N}N} w_n(t)\rt) = (1+ o_n(1)) e^{w_n(t)^{\frac N{N-1}}}
\]
for $t\geq a_n$ as $n\to \infty$. Hence
\begin{equation}\label{eq:A0}
\int_{a_n}^\infty F\lt(\alpha_{\beta,N}^{\frac{1-N}N} w_n(t)\rt) e^{-t} dt = (1+o_n(1) \int_{a_n}^\infty  e^{w_n(t)^{\frac N{N-1}}-t} dt,
\end{equation}
as $n\to \infty$. To proceed, we need the following result of Carleson and Chang \cite{CC1986}
\begin{lemma}\label{CC}
Let $a > 0$ and $\de >0$ be given numbers. Suppose that $\int_a^\infty |w'(t)|^N dt \leq \de$, then
\begin{align*}
\int_a^\infty &e^{w(t)^{\frac N{N-1}} -t} dt\\
& \leq \frac1{1 -\delta^{\frac1{N-1}}} \exp\lt(w(a)^{\frac N{N-1}} \lt[1 + \frac 1{N-1} \frac{\de}{(1 -\de^{\frac1{N-1}})^{N-1}}\rt]-a\rt) e^{1+\frac12 +\cdots+ \frac1{N-1}}.
\end{align*}
\end{lemma}
\noindent Applying Lemma \ref{CC} to $w_n$ with $a = a_n$ and 
\[
\de =\de_n: = \int_{a_n}^\infty \lt(|w_n'(t)|^N + \frac1{(1-\beta)^NN^N} |w_n(t)|^N e^{-\frac{t}{1-\beta}}\rt) dt,
\]
we get
\begin{equation}\label{eq:A1}
\int_{a_n}^\infty e^{w_n(t)^{\frac N{N-1}} -t} dt \leq \frac1{1 -\delta_n^{\frac1{N-1}}} e^{K_n} e^{1 + \frac12 + \cdots+ \frac1{N-1}},
\end{equation}
with 
\[
K_n := w_n(a_n)^{\frac N{N-1}} \lt[1 + \frac 1{N-1} \frac{\de_n}{(1 -\de_n^{\frac1{N-1}})^{N-1}}\rt]-a_n.
\]

We next use the following result proved in Appendix (see Lemma \ref{gamma} below): for $a >0$ and $b >0$, let $S_{N,\beta,a,b}$ be the set of functions $u \in W^{1,N}(-\infty,a)$ such that $\lim\limits_{t\to -\infty} u(t) =0$ and 
\[
\int_{-\infty}^a \lt(|u'(t)|^N + \frac1{(1-\beta)^NN^N} |u(t)|^N e^{-\frac{t}{1-\beta}}\rt) dt = b,
\]
then the supremum
\[
\sup\{ \|u\|_\infty^{\frac N{N-1}} : u \in S_{N,\beta,a,b}\}
\]
is attained by a function $w$ such that 
\[
\|w\|_\infty^{\frac N{N-1}} = w(a)^{\frac{N}{N-1}} = b^{\frac1{N-1}}a + (1-\beta)\alpha_N A_0 b^{\frac1{N-1}} + O(e^{-\frac a{N(1-\beta)}} a^N).
\]
Applying this result for $a_n$ and $b_n = 1-\de_n$, we get
\[
a_n - 2\ln a_n = w_n(a_n)^{\frac N{N-1}} \leq a_n (1-\de_n)^{\frac1{N-1}} +(1-\beta)\alpha_N A_0 (1-\de_n)^{\frac1{N-1}} + O(e^{-\frac {a_n}{N(1-\beta)}} a_n^N),
\]
or equivalently
\[
1 - (1-\de_n)^{\frac1{N-1}} = \frac{2 \ln a_n}{a_n} +(1-\beta)\alpha_N A_0 (1-\de_n)^{\frac1{N-1}} \frac1{a_n} + O(e^{-\frac {a_n}{N(1-\beta)}} a_n^{N-1}).
\]
It is easy to check that $1 -(1-\de_n)^{\frac1{N-1}} \geq \de_n /(N-1)$, hence
\[
\frac{\delta_n}{N-1}\lt(1+ \frac{(1-\beta)\alpha_N A_0}{a_n}\rt) \leq \frac{2 \ln a_n}{a_n} +\frac{(1-\beta)\alpha_N A_0}{a_n} + O(e^{-\frac {a_n}{N(1-\beta)}} a_n^{N-1}).
\]
We then get
\[
\de_n \leq \frac{2(N-1) \ln a_n}{a_n} + \frac{(N-1)(1-\beta)\alpha_N A_0}{a_n} + O\lt(\frac{\ln a_n}{a_n^2}\rt).
\]
This shows that $\de_n \to 0$ as $n\to \infty$. Plugging the estimate for $\de_n$ into $K_n$ we get
\begin{align*}
K_n&= a_n\lt[\lt(1-\frac{2\ln a_n}{a_n}\rt) \lt(1 + \frac{\de_n}{N-1} + O(\de_n^{\frac N{N-1}}) \rt) -1\rt]\\
&\leq a_n\lt[\lt(1-\frac{2\ln a_n}{a_n}\rt) \lt(1 + \frac{2\ln a_n}{a_n} +\frac{(1-\beta)\alpha_N A_0}{a_n} + O\lt(\lt(\frac{\ln a_n}{a_n}\rt)^{\frac N{N-1}}\rt) \rt) -1 \rt]\\
&= (1-\beta)\alpha_N A_0 + O\lt(\lt(\frac{\ln a_n}{a_n^{\frac1N}}\rt)^{\frac N{N-1}}\rt).
\end{align*}
Hence $\lim\limits_{n\to \infty} K_n \leq (1-\beta)\alpha_N A_0$. Letting $n\to \infty$ in \eqref{eq:A1}, we get
\[
\lim_{n\to \infty} \int_{a_n}^\infty e^{w_n(t)^{\frac N{N-1}} -t} dt \leq  e^{(1-\beta)\alpha_N A_0 + 1 + \frac12 + \cdots+ \frac1{N-1}}.
\]
This together with \eqref{eq:A0} proves \eqref{eq:enough}. Thus we have shown that
\[
\lim_{n\to \infty} \int_{\R^N} |x|^{-\beta N} F(u_n) dx \leq \frac1{1-\beta} |B_1| e^{(1-\beta)\alpha_N A_0 + 1 + \frac12 + \cdots + \frac1{N-1}},
\]
for any (RNCS) sequence $u_n$. Hence 
\begin{equation}\label{eq:A3}
d_{ncl}(N,\beta,F) \leq \frac1{1-\beta} |B_1| e^{(1-\beta)\alpha_N A_0 + 1 + \frac12 + \cdots + \frac1{N-1}}.
\end{equation}


We next construct a (RNCS) sequence $u_n$ such that
\begin{equation}\label{eq:***}
\lim_{n\to \infty} \int_{\R^N} |x|^{-N\beta} F\lt(u_n\rt) dx \geq \frac1{1-\beta} |B_1| e^{(1-\beta)\alpha_N A_0 + 1 + \frac12 + \cdots + \frac1{N-1}}.
\end{equation}
This will gives 
\begin{equation}\label{eq:A4}
d_{ncl}(N,\beta,F) \geq \frac1{1-\beta} |B_1| e^{(1-\beta)\alpha_N A_0 + 1 + \frac12 + \cdots + \frac1{N-1}}.
\end{equation}
Combining \eqref{eq:A3} and \eqref{eq:A4}, we obtain 
\[
d_{ncl}(N,\beta,F) = \frac1{1-\beta} |B_1| e^{(1-\beta)\alpha_N A_0 + 1 + \frac12 + \cdots + \frac1{N-1}}
\]
as wanted.

Let $\{\ep_n\}$ be a sequence of positive number such that $\ep_n \to 0$ as $n \to \infty$. Denote $R_n = (-\ln \ep_n)^{\frac1{1-\beta}}$ and define the function $u_n$ by
\begin{equation}\label{eq:unsequence}
u_n =
\begin{cases}
c_n + \frac{1}{c_n^{\frac1{N-1}}} \lt(-\frac{N-1}{\al_{\be,N}} \ln \lt(1+ b_N\lt(\frac{|x|}{\ep_n}\rt)^{\frac{N}{N-1}(1-\beta)}\rt)  + A_n\rt) &\mbox{if $|x|\leq R_n\ep_n$,}\\
\frac1{c_n^{\frac1{N-1}}} G(x) &\mbox{if $|x|> R_n\ep_n $}
\end{cases}
\end{equation}
where $b_N = \al_N/(N^{\frac N{N-1}}(1-\beta)^{\frac1{N-1}})$, and $A_n, c_n$ are constant depending on $n$ and $\beta$ to be determined such that $u_n \in W^{1,N}(\mathbb R^N)$ and $\|\na u_n\|_N^N + \|u\|_N^N =1$. These functions was introduced by Li and Ruf \cite{LiRuf2008} when $\beta =0$ and by Li and Yang \cite{LiYang} for $\beta \in (0,1)$.

It was computed in \cite{LiRuf2008,LiYang} that
\[
c_n^{\frac{N}{N-1}} = -\frac N{\al_N} \ln \ep_n + A_0 -\frac{N-1}{\al_{\be,N}} \sum_{k=1}^{N-1} \frac1k + \frac1{\al_{\be,N}} \ln \frac{\om_{N-1}}{N(1-\beta)} + O((-\ln \ep_n)^{-\frac N{N-1}}),
\]
and 
\[
A_n = -\frac{N-1}{\al_{\be,N}} \sum_{k=1}^{N-1} \frac1k + O((-\ln \ep_n)^{-\frac N{N-1}}).
\]
Hence $c_n \to \infty$ as $n\to 0$. Note that $R_n\ep_n \to 0$ as $n\to \infty$, hence $u_n$ is (RNCS). Moreover, again by \cite{LiRuf2008,LiYang}, we have
\begin{align}\label{eq:geq3}
\int_{\R^N}& |x|^{-N\beta} \Phi_N(\al_{\beta,N} u_n^{\frac N{N-1}}) dx \notag\\
&\qquad \geq \frac{|B_1|}{1-\beta}e^{\al_{\beta,N}A_0 + 1 + \frac12 +\cdots+\frac1{N-1}} + \frac{\al_{\beta,N}^{N-1}}{(N-1)! c_n^{\frac N{N-1}}} \lt( \int_{\R^N} \frac{|G|^N}{|x|^{\beta N}} dx + o_n(1) \rt),
\end{align}
where $o_n(1) \to 0$ as $n\to \infty$.

By the assumption (F3) on $F$ and the criticality of $F$, we can write 
\[
F(t) = \Phi_N(\al_{\beta,N} |t|^{\frac N{N-1}}) -G(t)
\]
with $G\geq 0$ satisfying the conditions (F1)-F(4) and being subcritical. Hence, we get
\[
\lim_{n\to \infty} \int_{\R^N} |x|^{-\beta N} G(u_n) dx =0.
\]
This together with \eqref{eq:geq3} implies \eqref{eq:***}.

\section{Proof of Theorem \ref{Maximizers}}
This section is devoted to prove Theorem \ref{Maximizers}. Let $u_n$ be a maximizer sequence for $MT(N,\beta,F)$. By the rearrangement argument, we can assume that $u_n$ is decreasing symmetric, radial function. We can assume further that $u_n \rightharpoonup u_0$ weakly in $W^{1,N}(\mathbb R^N)$, $u_n \to u_0$ in $L^p_{\rm loc}(\R^N)$ for any $p < \infty$ and $u_n \to u_0$ a.e. in $\R^N$. We first exclude the concentrating or vanishing behavior of the sequence $u_n$. We will need the following lemma which gives a lower bound for $MT(N,\beta,F)$.
\begin{lemma}\label{lowerboundMT}
Let $F$ be function satisfying the condition (F1)-(F4) above. Then
\[
MT(N,\beta,F) > \max\{d_{nvl}(N,\beta,F), d_{ncl}(N,\beta,F)\}
\]
under the conditions on $F$ in Theorem \ref{Maximizers}.
\end{lemma}
\begin{proof}
If $\beta \in (0,1)$ and $F$ is subcritical then $d_{nvl}(N,\beta,F) = d_{ncl}(N,\beta,F) =0$. Hence our conclusion is trivial.

If $\beta \in (0,1)$ and $F$ is critical. By Theorem \ref{NVSlimit} and Theorem \ref{NCSlimit} we have 
\[
d_{nvl}(N,\beta,F) =0\qquad\text{\rm and }\qquad d_{ncl}(N,\beta,F) = \frac{|B_1|}{1-\beta} e^{(1-\beta)\al_N A_0 + 1+\frac12 + \cdots+ \frac1{N-1}}.
\]
Let $u_n$ be the sequence defined by \eqref{eq:unsequence}, we have
\begin{equation}\label{eq:unintegration}
\int_{\R^N} |u_n|^N dx =\frac{1}{c_n^{\frac N{N-1}}} \lt(\int_{\R^N} \frac{|G|^N}{|x|^{N\beta}} dx + o_n(1)\rt).
\end{equation}
Combining \eqref{eq:unintegration}, \eqref{eq:geq3} and \eqref{eq:lowercond}, we have
\begin{align}\label{eq:geq4}
\int_{\R^N} &|x|^{-N \beta} F(u_n) dx \notag\\
&\geq \frac{|B_1|}{1-\beta}e^{\al_{\beta,N}A_0 + 1 + \frac12 +\cdots+\frac1{N-1}} + \frac{1}{ c_n^{\frac N{N-1}}} \lt[\lt(\frac{\al_{\beta,N}^{N-1}}{(N-1)!}-\lambda\rt) \int_{\R^N} \frac{|G|^N}{|x|^{\beta N}} dx + o_n(1) \rt].
\end{align}
This implies our conclusion for $n$ large enough since $\lam < \alpha_{\beta,N}^{N-1}/(N-1)!$.

If $\beta =0$, and $F$ is subcritical. We know that $d_{ncl}(N,0,F) =0$ and $d_{nvl}(N,0,F) =C(F)$. Hence, it is enough to prove $MT(N,0,F) > d_{nvl}(N,0,F)$. Fix a function $v \in W^{1,N}(\R^N)$ which is bounded, spherically symmetric and non-increasing such that $\|v\|_N =1$. For $t > 0$, define the new function $w_t$ by
\[
w_t(x) = \frac{t^{1/N} v(t^{1/N}x)}{(1+ t \|\na v\|_N^N)^{\frac1N}},\qquad x\in \R^N.
\]
We then have $\|w_t\|_N^N + \|\na w_t\|_N^N =1$ and $w_t \to 0$ uniformly on $\R^N$. By \eqref{eq:near0}, we get
\[
F(w_t) = C(F) |w_t|^{N} + \sum_{k=N}^{2(N-1)}C_k |w_t|^{\frac{Nk}{N-1}} + o(|w_t|^{2N}).
\]
Hence
\begin{align}\label{eq:asymptotict}
\int_{\R^N} & F(w_t) dx\notag\\
&= \frac{C(F)}{1+ t\|\na v\|_N^N} + \sum_{k=1}^{2(N-1)}C_k \frac{t^{\frac k{N-1} -1} \|v\|_{\frac{Nk}{N-1}}^{\frac{Nk}{N-1}}}{(1+ t\|\na v\|_N^N)^{\frac k{N-1}}} + o(t)\notag\\
&=C(F) + tC(F) \|\na v\|_N^N\lt(\frac{C_{2(N-1)}}{C(F)} \frac{\|v\|_{2N}^{2N}}{\|\na v\|_N^N} -1\rt) + \sum_{k=N}^{2N-3} t^{\frac k{N-1}-1}C_k \|v\|_{\frac{Nk}{N-1}}^{\frac{Nk}{N-1}} + o(t)
\end{align}
Consequently, if $C_k >0$ for some $k =N,\ldots,2N-3$ we then have $MT(N,0,F) > C(F) = d_{nvl}(N,0,F)$ for $t >0$ small enough. If $C_k=0$ for any $k =N,\ldots,2N-3$, by taking $v$ to be the maximizer of the Gagliardo--Nirenberg--Sobolev inequality \eqref{eq:GNSineq}, we then have
\[
\int_{\R^N} F(w_t) dx = C(F) + tC(F) \|\na v\|_N^N\lt(\frac{C_{2(N-1)}}{C(F)} B_N -1\rt) + o(t).
\]
Hence, by choosing $t >0$ small enough, we get $MT(N,0,F) > C(F) = d_{nvl}(N,0,F)$ if $C_{2(N-1)}> C(F)/B_N$.

If $\beta =0$ and $F$ is critical. Note that under the condition \eqref{eq:lowercond}, the estimate \eqref{eq:geq4} still holds for $\beta =0$. This implies $MT(N,0,F) > d_{ncl}(N,0,F)$. Also, under the condition of $F$ near zero the estimate \eqref{eq:asymptotict} still holds. Repeating the argument in the subcritical case, we obtain $MT(N,0,\beta) > d_{nvl}(N,0,\beta)$ as wanted.
\end{proof}

Since $u_n$ is a maximizer sequence, we have
\[
\lim_{n\to \infty} \int_{\R^N} |x|^{-N\beta} F(u_n) dx =MT(N,\beta,F).
\]
This together Lemma \ref{CCLions}, Theorem \ref{NVSlimit}, Theorem \ref{NCSlimit} and Lemma \ref{lowerboundMT} implies that $u_0\not\equiv 0$. Again by Lemma \ref{CCLions}, we get
\begin{equation}\label{eq:CC*}
\limsup_{n\to\infty} \int_{\R^N} |x|^{-N\beta} \Phi_N(\alpha_{\beta,N} p u_n^{\frac N{N-1}}) dx < \infty,
\end{equation}
for any $1 < p < (1 -\|u_0\|_N^N -\|\na u_0\|_N^N)^{-\frac1{N-1}}$. We have following two cases:

$\bullet$ \emph{Case 1:} $\beta \in (0,1)$. For any $R >0$, we have $u_n(r) \leq C R^{-\frac{N-1}N}$ for any $r \geq R$ with $C$ depending only on $N$ by Lemma \ref{radial}. Thus $u_n(r) = O(R^{-\frac{N-1}N})$ uniformly in $n$ for any $r \geq R$. Hence
\[
F(u_n(r)) = (C(F) + o_R(1)) u_n(r)^N
\]
uniformly on $n$ for any $r \geq R$, here $o_R(1) \to 0$ as $R\to \infty$. Consequently, we get
\[
\int_{B_R^c} |x|^{-N \beta} F(u_n) dx = (C(F) + o_R(1))\int_{B_R^c} \frac{u_n^N}{|x|^{N \beta}} dx =o_R(1),
\]
uniformly in $n$, since $\beta \in (0,1)$. On $B_R$ we have
\[
F(u_n)^p \leq \Phi_N(\al_{\beta,N} u_n^{\frac N{N-1}})^p \leq e^{\alpha_{\beta,N} p u_n^{\frac N{N-1}}},
\]
for any $p \geq 1$. Repeating the argument in proof of \eqref{eq:R0} with the help of \eqref{eq:CC*}, we obtain that $F(u_n)$ is bounded in $L^r(B_R, |x|^{-N\beta} dx)$ for some $r >1$, hence
\[
\lim_{n\to \infty} \int_{B_R} |x|^{-N\beta} F(u_n) dx = \int_{B_R} |x|^{-N\beta} F(u_0) dx,
\]
for any $R >0$. Consequently, 
\[
MT(N,\beta,F) = \lim_{n\to\infty} \int_{\R^N} |x|^{-N\beta} F(u_n) dx = \int_{\R^N} |x|^{-N\beta} F(u_0) dx.
\]
This equality together with the strictly increasing monotonicity of the function $F$ implies that $\|u_0\|_N^N + \|\na u_0\|_N^N =1$, hence $u_0$ is a maximizers for $MT(N,\beta,F)$.

$\bullet$ \emph{Case 2:} $\beta =0$. Let $a =\lim\limits_{n\to \infty} \|u_n\|_N^N$ then $\|u_0\|_N^N \leq a \leq 1$. As in the case $\beta \in (0,1)$, we have
\[
\int_{B_R^c} F(u_n) dx = C(F) \int_{B_R^c} u_n^N dx + o_R(1) =C(F) \lt(\|u_n\|_N^N -\int_{B_R} u_n^N dx \rt) + o_R(1),
\]
uniformly in $n$, hence
\[
\lim_{R\to\infty} \lim_{n\to \infty} \int_{B_R^c} F(u_n) dx =C(F) (a -\|u_0\|_N^N).
\]
Similar in the case $\beta \in (0,1)$, we have
\[
\lim_{n\to \infty} \int_{B_R} F(u_n) dx = \int_{B_R} F(u_0) dx,
\]
for any $R >0$. Hence
\begin{align*}
MT(N,0,F) &= \lim_{n\to\infty} \int_{\R^N} F(u_n) dx\\
&=\lim_{n\to\infty}\lt(\int_{B_R} F(u_n) dx + \int_{B_R^c} F(u_n) dx\rt)\\
&=\int_{B_R} F(u_0) dx + \lim_{n\to\infty}\int_{B_R^c} F(u_n) dx,
\end{align*}
for any $R > 0$. Letting $R\to\infty$, we get
\begin{equation}\label{eq:AA*}
MT(N,0,F) = \int_{\R^N} F(u_0) dx + C(F)(a -\|u_0\|_N^N).
\end{equation}
Let $\tau^N = a/\|u_0\|_N^N \geq 1$, and define $v_0(x) = u_0(x/\tau)$. We have $\|v_0\|_N^N = \tau^N \|u_0\|_N^N = a$ and 
\[
\|\na v_0\|_N^N = \|\na u_0\|_N^N \leq \liminf_{n\to \infty} \|\na u_n\|_N^N.
\]
Hence $\|v_0\|_N^N + \|\na v_0\|_N^N \leq 1$, consequently we have
\[
MT(N,0,F)\geq \int_{\R^N} F(v_0) dx = \tau^N \int_{\R^N} F(u_0) dx.
\]
The preceding estimate and \eqref{eq:AA*} imply
\[
(\tau^N -1) \lt(\int_{\R^N} F(u_0) dx - C(F) \|u_0\|_N^N\rt) \leq 0.
\]
If $\tau >1$ then $\int_{\R^N} F(u_0) dx \leq C(F) \|u_0\|_N^N$. This together \eqref{eq:AA*} implies 
\[
MT(N,0,F) \leq C(F) a \leq C(F)
\]
which is impossible by Lemma \ref{lowerboundMT}. Hence $\tau =1$, again by \eqref{eq:AA*} we get
\[
MT(N,0,F) = \int_{\R^N} F(u_0) dx.
\]
This equality together with the strictly increasing monotonicity of $F$ implies that $\|u_0\|_N^N + \|\na u_0\|_N^N =1$, hence $u_0$ is a maximizers for $MT(N,0,F)$.

This finishes the proof of Theorem \ref{Maximizers}.

\section{An auxiliary variational problem}
In this section, we study the variational problem
\begin{equation}\label{eq:variationalp}
\sup\,\{ \|u\|_\infty^{\frac N{N-1}} : u \in S_{N,\beta,a,b}\}.
\end{equation}
Recall that $S_{N,\beta,a,b}$ be the set of functions $u \in W^{1,N}(-\infty,a)$ such that $\lim\limits_{t\to -\infty} u(t) =0$ and 
\[
\int_{-\infty}^a \lt(|u'(t)|^N + \frac1{(1-\beta)^NN^N} |u(t)|^N e^{-\frac{t}{1-\beta}}\rt) dt = b,
\]
where $a,b >0$ are given number. Note that the embedding $S_{N,\beta,a,b} \hookrightarrow L^\infty(-\infty,a)$ is compact, then the supremum in \eqref{eq:variationalp} is attained by a function $u_a \in S_{N,\beta,a,b}$ such that
\begin{equation}\label{eq:maximizers}
\|u_a\|_\infty^{\frac N{N-1}} = \sup \, \{\|u\|_\infty^{\frac N{N-1}}\, :\, u\in S_{N,\beta,a,b}\}.
\end{equation}
In the sequel, we will determine the value of this supremum by identifying the maximizer $u_a$. The natural way to do this is based on the Euler--Lagrange equation associated to \eqref{eq:variationalp}, but we emphasize the difficulty here that the functional $\psi: u \to \|u\|_\infty^{\frac N{N-1}}$ is not differentiable.
However this functional is convex and hence its subdifferential $\pa\psi$ exists. We will briefly recall this notion and derive the Euler--Lagrange equation associated to \eqref{eq:variationalp}. By analyzing this equation, we are able to determine the maximizer $u_a$ above. This strategy was already used in \cite{Ruf2005} which goes back to the earlier paper of De Figueiredo and Ruf \cite{deG91}.

\begin{definition}
Let $E$ be a Banach space and $\psi: E \to \R$ be continuous and convex functional. Then the subdifferential $\pa\psi(u)$ of $\psi$ at $u\in E$ is the subset of the dual space $E'$ characterized by
\[
\mu_u \in \pa\psi(u) \Longleftrightarrow \psi(u+v) - \psi(u) \geq \la \mu_u,v\ra,\qquad\forall\, v\in E,
\]
where $\la\cdot,\cdot\ra$ denotes the duality pairing between $E$ and $E'$. An element $\mu_u\in \pa\psi(u)$ is called a subgradient of $\psi$ at $u$.
\end{definition}
Let us recall the following two lemmas in \cite{Cassani} whose proofs are slight modifications of the proof of Lemma $2.2$, Lemma $2.3$ and Corollary $2.4$ in \cite{deG91}.
\begin{lemma}\label{l1}
Let $\psi: E\to \R$ be a convex and continuous functional. Assume that $\psi \geq 0$ and that for $q\geq 1$ we have $\psi(t x) = t^q \psi(x)$ for any $x\in E$ and $t \geq 0$. Then $\mu_u \in \pa\psi(u)$ if and only if $\la \mu_u,u\ra = q \psi(u)$ and 
\[
\la \mu_u,v\ra \leq \la \mu_u,u\ra,\qquad \forall\, v\in \psi^u :=\{v\in E\, :\, \psi(v) \leq \psi(u)\}.
\]
\end{lemma}

\begin{lemma}\label{l2}
Let $\phi \in C^1(E;\R)$ be a functional such that 
\[
\la \phi'(v),v\ra = q \phi(v) \not=0, \qquad\forall\, v\in E\setminus\{0\},
\]
and let $\psi : E \to \R$ be a functional satisfying the hypotheses of Lemma \ref{l1}. If $w \in E$ is such that for some $A >0$,
\[
\psi(w) = \sup_{u\in E, \, \phi(u) =1} \psi(u) = \frac1A,
\]
then
\[
\phi'(w) \in \frac1{\psi(w)} \pa\psi(w) = A \pa \psi(w).
\]
\end{lemma}

We next apply Lemma \ref{l1} and Lemma \ref{l2} to 
\[
E = \{u\in W^{1,N}(-\infty,a): \lim_{t\to\ -\infty} u(t) =0\},
\]
\[
\phi(u) =  \int_{-\infty}^a \lt(|u'(t)|^N + \frac1{(1-\beta)^NN^N} |u(t)|^N e^{-\frac{t}{1-\beta}}\rt) dt,
\]
and
\[
\psi(u) =  \|u\|_\infty^N,
\]
and obtain the following result.
\begin{proposition}\label{p1}
Suppose that $w \in E$ satisfies
\[
\psi(w) = \sup\, \{\psi(u)\, :\, u\in E,\, \phi(u) =b\},
\]
then $w$ satisfies (weakly) the equation
\begin{equation}\label{eq:E-L1}
(-|w'(t)|^{N-2}w'(t))' + \frac1{(1-\beta)^N N^N} |w(t)|^{N-2} w(t) e^{-\frac t{1-\beta}} = \frac{b}{N\|w\|_\infty^N}\, \mu_w,
\end{equation}
where $\mu_w \in \pa\psi(w) \subset E'$.
\end{proposition}
Note that we can assume that $w$ is nonnegative. It remains to determine the subgradient $\mu_w$ in equation \eqref{eq:E-L1}. Following \cite[Lemmas $2.6-2.8$]{deG91}, we have the following proposition. 
\begin{proposition}\label{p2}
Let $K_w =\{t\in (-\infty,a]\,:\, w(t) = \|w\|_\infty\}$. Then
\begin{description}
\item (i) $\text{\rm supp}\, \mu_w \subset K_w$.
\item (ii) $K_w =\{a\}$.
\item (iii) $\mu_w = N\|w\|_\infty^{N-1} \de_a$ where $\de_a$ is Dirac function at $a$.
\end{description}
\end{proposition}
Thus, equation \eqref{eq:E-L1} becomes
\begin{equation}\label{eq:EL}
\begin{cases}
(-|w'(t)|^{N-2} w'(t))' + \frac1{(1-\beta)^N N^N} w(t)^{N-1} e^{-\frac t{1-\beta}} = \frac{b}{N\|w\|_\infty^N}\, \mu_w,&\mbox{$-\infty < t\leq a,$}\\
\lim\limits_{t\to -\infty} w(t) = 0.
\end{cases}
\end{equation}
By this, we conclude that
\begin{equation}\label{eq:EL1}
\begin{cases}
(-|w'(t)|^{N-2} w'(t))' + \frac1{(1-\beta)^N N^N} w(t)^{N-1} e^{-\frac t{1-\beta}} = 0,&\mbox{$-\infty < t< a,$}\\
\lim\limits_{t\to -\infty} w(t) = 0,
\end{cases}
\end{equation}
with the condition that
\begin{equation}\label{eq:EL2}
\int_{-\infty}^a \lt(|w'(t)|^N dt + \frac1{(1-\beta)^N N^N} w(t)^N e^{-\frac t{1-\beta}}\rt) dt = b.
\end{equation}
Applying elliptic estimates, we have $w \in C^{1,\al}((-\infty,a])$ for some $0< \al \leq 1$. Integrating equation \eqref{eq:EL1} on $(-\infty,t)$ with $t < a$, we get
\[
|w'(t)|^{N-2} w'(t) = \frac1{(1-\beta)^N N^N} w(t)^N\int_{-\infty}^t w(s)^{N} e^{-\frac s{1-\beta}} ds.
\]
This shows that $w'$ is positive and increasing.


It is easy to see that
\[
w(t) = \gamma(a,b) G(e^{-\frac t{N(1-\beta)}}),\qquad t\leq a,
\]
where $\gamma(a,b)$ is constant determined by \eqref{eq:EL2}.
\begin{lemma}\label{gamma}
It holds
\[
\gamma(a,b)^N =\lt(N(1-\beta)\rt)^{N-1} \frac{b}{\lt(-G'(e^{-\frac{a}{N(1-\beta)}}) e^{-\frac{a}{N(1-\beta)}}\rt)^{N-1} G(e^{-\frac{a}{N(1-\beta)}})}.
\]
Consequently, we get
\[
\|w\|_\infty^{\frac N{N-1}} = w(a)^{\frac N{N-1}} = b^{\frac1{N-1}} a + (1-\beta)\al_N A_0 b^{\frac1{N-1}} + O(e^{-\frac a{N(1-\beta)}} a^N).
\]
\end{lemma}
\begin{proof}
By an easy computation, we have
\[
w'(t) = -\gamma(a,b)\frac1{N(1-\beta)}G'(e^{-\frac{t}{N(1-\beta)}}) e^{-\frac{t}{N(1-\beta)}}.
\]
Thus
\begin{align*}
b&= \frac{\gamma(a,b)^N}{(N(1-\beta))^N} \int_{-\infty}^a \lt(|G'(e^{-\frac{t}{N(1-\beta)}})|^N + |G(e^{-\frac{t}{N(1-\beta)}})|^N\rt) e^{-\frac t{1-\beta}} dt\\
&=\frac{\gamma(a,b)^N}{(N(1-\beta))^{N-1}}(-G'(e^{-\frac{a}{N(1-\beta)}})e^{-\frac{a}{N(1-\beta)}})^{N-1} G(e^{-\frac{a}{N(1-\beta)}}),
\end{align*}
which implies the first conclusion.

Replacing the value of $\gamma(a,b)$ in to $w$, we obtain
\begin{align*}
\|w\|_\infty^{\frac N{N-1}} = w(a)^{\frac N{N-1}} &=N(1-\beta) b^{\frac1{N-1}} \frac{G(e^{-\frac{a}{N(1-\beta)}})}{-G'(e^{-\frac{a}{N(1-\beta)}})e^{-\frac{a}{N(1-\beta)}}}\\
&=b^{\frac1{N-1}}a + (1-\beta)\alpha_N A_0 b^{\frac1{N-1}} + O(e^{-\frac a{N(1-\beta)}} a^N),
\end{align*}
here we use the asymptotic behaviors of $G$ and $G'$ near origin (see \eqref{eq:formG} and \eqref{eq:formG'}).
\end{proof}

\section*{Acknowledgments}
This work is supported by the CIMI's postdoctoral research fellowship.


\begin{thebibliography}{99}
\bibitem{Adachi}
S. Adachi, and K. Tanaka, \emph{Trudinger type inequalities in $\R^N$ and their best constant\text}, Proc. Amer. Math. Soc., {\bf 128} (2000) 2051--2057.


\bibitem{AD2004}
Adimurthi, and O. Druet, \emph{Blow-up analysis in dimension $2$ and a sharp form of Trudinger--Moser inequality\text}, Comm. Partial Differ. Equ., {\bf 29} (2004) 295--322.

\bibitem{AS07}
Adimurthi, and K. Sandeep, \emph{A singular Moser--Trudinger embedding and its applications\text}, Nonlinear Differ. Equ. Appl., {\bf 13} (2007) 585--603.

\bibitem{AY10}
Adimurthi, and Y. Yang, \emph{An interpolation of Hardy inequality and Trudinger--Moser inequality in $\R^N$ and its applications\text}, Int. Math. Res. Not., IMRN {\bf 13} (2010) 2394--2426.


\bibitem{Brothers}
J. E. Brothers, and W. P. Ziemer, \emph{Minimal rearrangements of Sobolev functions\text}, J. Reine. Angew. Math., {\bf 348} (1988) 153--179.

\bibitem{Cao}
D. Cao, \emph{Nontrivial solution of semilinear elliptic equations with critical exponent in $\R^2$\text}, Comm. Partial Differential Equations, {\bf 17} (1992) 407--435.

\bibitem{CC1986}
L. Carleson, and S. Y. A. Chang, \emph{On the existence of an extremal function for an inequality of J. Moser\text}, Bull. Sci. Math., {\bf 110} (1986) 113--127.

\bibitem{Cassani}
D. Cassani, B. Ruf, and C. Tarsi, \emph{Group invariance and Pohozaev identity in Moser--type inequalities\text}, Comm. Contemp. Math., {\bf 15} (2013) 20 pages.

\bibitem{deG91}
D. De Figueiredo, and B. Ruf, \emph{On a superlinear Sturm--Liouville equation and a related bouncing problem\text}, J. Reine Angew. Math., {\bf 421} (1991) 1--22.

\bibitem{deG02}
D. De Figueiredo, J. M. do \'O, and B. Ruf, \emph{On an inequality by N. Trudinger and J. Moser and related elliptic equations\text}, Comm. Pure Appl. Math., {\bf 55} (2002) 1--18.

\bibitem{doO97}
J. M. do \'O, \emph{$N-$Laplacian equations in $\R^N$ with critical growth\text}, Abstr. Appl. Anal., {\bf 2} (1997) 301--315.

\bibitem{doO2014}
J. M. do \'O, M. de Souza, E. de Medeiros, and U. B. Severo, \emph{An improvement for the Trudinger--Moser inequality and applications\text}, J. Differential Equations, {\bf 256} (2014) 1317--1349.

\bibitem{doO2014*}
J. M. do \'O, and M. de Souza,  \emph{A sharp Trudinger--Moser type inequality in $\R^2$\text}, Trans. Amer. Math. Soc., {\bf 366} (2014) 4513--4549.

\bibitem{doO2015*}
J. M. do \'O, M. de Souza, E. de Medeiros, U. B. Severo, \emph{Critical points for a functional involving critical growth of Trudinger--Moser type\text}, Potential Anal., {\bf 42} (2015) 229--246.

\bibitem{doO2015}
J. M. do \'O, and M. de Souza, \emph{A sharp inequality of Trudinger--Moser type and extremal functions in $H^{1,n}(\R^n)$\text}, J. Differential Equations, {\bf 258} (2015) 4062--4101.

\bibitem{doO2016}
J. M. do \'O, and M. de Souza, \emph{Trudinger--Moser inequality on the whole plane and extremal functions\text}, Commun. Contemp. Math., {\bf 18} (2016) 32 pp.

\bibitem{doO16}
J. M. do \'O, F. Sani, and C. Tarsi, \emph{Vanishing--concentration--compactness alternative for the Trudinger--Moser inequality in $\R^N$\text}, Commun. Contemp. Math., {\bf 19} (2016) 27pp.

\bibitem{Druet}
O. Druet, E. Hebey, and F. Robert, \emph{Blow-up theory for elliptic PDEs in Riemannian geometry\text}, Math. Notes, vol. 45, Princeton University press, Princeton, NJ, 2004.

\bibitem{Flucher1992}
M. Flucher, \emph{Extremal functions for the Trudinger-Moser inequality in $2$ dimensions\text} 
Comment. Math. Helv., {\bf 67} (1992) 471--497

\bibitem{Ishi}
M. Ishiwara, \emph{Existence and nonexistence of maximizers for variational problems associated with Trudinger--Moser inequalities in $\mathbb R^N$\text}, Math. Ann., {\bf 351} (2011) 781--804.

\bibitem{LamLuhei}
N. Lam, and G. Lu, \emph{Sharp Moser--Trudinger inequality on the Heisenberg group at the critical case and applications\text}, Adv. Math., {\bf 231} (2012) 3259--3287.

\bibitem{LamLu}
N. Lam, and G. Lu, \emph{A new approach to sharp Moser--Trudinger and Adams type inequalities: a rearrangement--free argument\text}, J. Differential Equations, {\bf 255} (2013) 298--325.

\bibitem{LamLuZhang}
N. Lam, G. Lu, and L. Zhang, \emph{Equivalence of critical and subcritical sharp Trudinger--Moser--Adams inequalities\text}, Rev. Mat. Iberoam., (to appear).

\bibitem{LamLuZhanga}
N. Lam, G. Lu, and L. Zhang, \emph{Existence and nonexistence of extremal functions for sharp Trudinger--Moser inequalities\text}, preprint.

\bibitem{Li2001}
Y. Li, \emph{Moser--Trudinger inequaity on compact Riemannian manifolds of dimension two\text}, J. Partial Differ. Equa., {\bf 14} (2001) 163-192.

\bibitem{Li2005}
Y. Li, \emph{Extremal functions for the Moser-Trudinger inequalities on compact Riemannian manifolds\text}, Sci. China Ser. A, {\bf 48} (2005) 618--648.

\bibitem{LiRuf2008}
Y. Li, and B. Ruf, \emph{A sharp Trudinger-Moser type inequality for unbounded domains in $\R^n$\text}, Indiana Univ. Math. J., {\bf 57} (2008) 451--480.

\bibitem{LiYang}
X. Li, and Y. Yang, \emph{Extremal functions for singular Trudinger--Moser inequalities in the entire Euclidean space\text}, preprint, arXiv:1612.08241v1.

\bibitem{Lin1996}
K. Lin, \emph{Extremal functions for Moser's inequality\text}, Trans. Amer. Math. Soc., {\bf 348} (1996) 2663--2671.

\bibitem{Lions1985}
P. L. Lions, \emph{The concentration-compactness principle in the calculus of variations. The limit case. II\text}, Rev. Mat. Iberoamericana, {\bf 1} (1985) 45-121.

\bibitem{LuYang}
G. Lu, and Y. Yang, \emph{Adams' inequalities for bi-Laplacian and extremal functions in dimension four\text}, Adv. Maths., {\bf 220} (2009) 1135--1170.

\bibitem{LuZhu}
G. Lu, and M. Zhu, \emph{A sharp Moser--Trudinger type inequality involving $L^n$ norm in the entire space $\R^n$\text}, preprint, arXiv:1703.00901v1.

\bibitem{M1970}
J. Moser, \emph{A sharp form of an inequality by N. Trudinger\text}, Indiana Univ. Math. J., {\bf 20} (1970/71) 1077--1092.



\bibitem{Nguyen2017}
V. H. Nguyen, \emph{Extremal functions for the Moser--Trudinger inequality of Adimurthi--Druet type in $W^{1,N}(\mathbb R^N)$\text}, preprint, arXiv:1702.07970v1.

\bibitem{Nguyen17}
V. H. Nguyen, \emph{Existence of maximizers for the critical sharp Moser--Trudinger type inequality in $\mathbb R^N$\text}, preprint.

\bibitem{P1965}
S. I. Poho${\rm \check{z}}$aev, \emph{On the eigenfunctions of the equation $\Delta u + \lambda f(u) = 0$\text}, (Russian), Dokl. Akad. Nauk. SSSR, {\bf 165} (1965) 36-39.

\bibitem{Ruf2005}
B. Ruf, \emph{A sharp Trudinger-Moser type inequality for unbounded domains in $\R^2$\text}, J. Funct. Anal., {\bf 219} (2005) 340--367.

\bibitem{Struwe}
M. Struwe, \emph{Critical points of embeddings of $H_0^{1,n}$ into Orlicz spaces\text}, Ann. Inst. H. Poincar\'e, Analyse Non Lin\'eaire, {\bf 5} (1988) 425--464.

\bibitem{Tintarev}
C. Tintarev, \emph{Trudinger--Moser inequality with remainder terms\text}, J. Funct. Anal., {\bf 266} (2014) 55--66.

\bibitem{T1967}
N. S. Trudinger, \emph{On imbedding into Orlicz spaces and some applications\text}, J. Math. Mech., {\bf 17} (1967) 473-483.

\bibitem{WY2012}
G. Wang, and D. Ye, \emph{A Hardy--Moser--Trudinger inequality\text}, Adv. Math., {\bf 230} (2012) 294--320.


\bibitem{Yang06}
Y. Yang, \emph{Extremal functions for Moser--Trudinger inequalities on $2-$dimensional compact Riemannian manifolds with boundary\text}, Int. J. Math., {\bf 17} (2006) 313--330.

\bibitem{Yang06*}
Y. Yang, \emph{A sharp form of Moser--Trudinger inequality in high dimension\text}, J. Funct. Anal., {\bf 239} (2006) 100--126.

\bibitem{Yang09}
Y. Yang, \emph{A sharp form of Moser--Trudinger inequality on compact Riemannian surface\text}, Trans. Amer. Math. Soc., {\bf 359} (2007) 5761--5776.

\bibitem{Yang12}
Y. Yang, \emph{Existence of positive solutions to quasi-linear elliptic equations with exponential growth in the whole Euclidean space\text}, J. Funct. Anal., {\bf 262} (2012) 1679--1704.

\bibitem{Yang12a}
Y. Yang, \emph{Adams type inequalities and related elliptic partial differential equations in dimension four\text}, J. Differential Equations, {\bf 252} (2012) 2266--2295.

\bibitem{Yang12b}
Y. Yang, \emph{Trudinger--Moser inequalities on complete noncompact Riemannian manifolds\text}, J. Funct. Anal., {\bf 263} (2012) 1894--1938.

\bibitem{Yang2015}
Y. Yang, \emph{Extremal functions for Trudinger--Moser inequalities of Adimurthi--Druet type in dimension two\text}, J. Differential Equations, {\bf 258} (2015) 3161--3193.

\bibitem{Yang2017}
Y. Yang, and X. Zhu, \emph{Blow-up analysis concerning singular Trudinger--Moser inequalities in dimension two\text}, J. Funct. Anal., {\bf 272} (2017) 3347--3374.

\bibitem{Y1961}
V. I. Yudovi${\rm \check{c}}$, \emph{Some estimates connected with integral operators and with solutions of elliptic equations\text}, (Russian), Dokl. Akad. Nauk. SSSR, {\bf 138} (1961) 805-808.
\end{thebibliography}
\end{document}